\documentclass[leqno,12pt]{amsart}
\usepackage{setspace}
\usepackage[english]{babel}
\usepackage{bussproofs}
\usepackage{csquotes}
\usepackage{dirtytalk}
\usepackage{mathtools}
\usepackage{mathrsfs}
\usepackage{latexsym}
\usepackage{enumitem}
\usepackage{amsthm}
\usepackage{amssymb}
\DeclareMathAlphabet{\mathpzc}{OT1}{pzc}{m}{it}
\usepackage[latin1]{inputenc}
\usepackage{afterpage}

\usepackage{bbm}
\usepackage[rgb,dvipsnames]{xcolor}
\usepackage{tikz} 
\usetikzlibrary{decorations.text}
\usetikzlibrary{arrows,arrows.meta,hobby}
\usetikzlibrary{shapes.misc, positioning}

\usepackage{tikz-cd}

\usetikzlibrary{cd}

\newtheorem{theorem}{Theorem}[section]
\newtheorem{maintheorem}{Main Theorem}[section]

\newtheorem{lemma}[theorem]{Lemma}
\newtheorem{proposition}[theorem]{Proposition}
\newtheorem{corollary}[theorem]{Corollary}
\newtheorem{observation}[theorem]{Observation}
\newtheorem{fact}[theorem]{Fact}
\newtheorem{claim}[theorem]{Claim}
\theoremstyle{definition}
\newtheorem{definition}[theorem]{Definition}

\theoremstyle{remark}
\newtheorem{remark}{Remark}

\newtheorem{question}{Question}

\def\hook{\upharpoonright}
\def\forces{\Vdash}
\def\barN{\bar{N}}
\def\barG{\bar{G}}

\def\ZFC{\mathsf{ZFC}}

\def\MM{\mathsf{MM}}
\def\PFA{\mathsf{PFA}}
\def\SubPFA{\mathsf{SubPFA}}
\def\MA{\mathsf{MA}}
\def\SCFA{\mathsf{SCFA}}

\def\mfc{\mathfrak{c}}

\def \mfd{\mathfrak{d}}

\def\CH {\mathsf{CH}}
\def\Q{\mathbb Q}
\def\P{\mathbb P}

\def\T{\mathbb T}
\def\R{\mathbb R}

\usepackage[margin=1.3in]{geometry}

\begin{document}

\title{Separating Subversion Forcing Axioms}

\author[Sakai]{Hiroshi Sakai}
\address[H. Sakai]{Graduate School of Mathematical Sciences, The University of Tokyo, 3-8-1 Komaba, Meguro-ku, Tokyo, 153-8914, Japan}
\email{hsakai@ms.u-tokyo.ac.jp}

\author[Switzer]{Corey Bacal Switzer}
\address[C.~B.~Switzer]{Institut f\"{u}r Mathematik, Kurt G\"odel Research Center, Universit\"{a}t Wien, Kolingasse 14-16, 1090 Wien, AUSTRIA}
\email{corey.bacal.switzer@univie.ac.at}

\thanks{\emph{Acknowledgments:} The first author would like to thank JSPS for the support through grant numbers 21K03338 and 24K06828. The second author's research was funded in whole or in part by the Austrian Science Fund (FWF) the following grants: 10.55776/Y1012, I4513, and 10.55776/ESP548}

\date{}

\maketitle

\begin{abstract}
We study a family of variants of Jensen's
\emph{subcomplete forcing axiom}, $\SCFA$ and \emph{subproper forcing axiom}, $\SubPFA$. Using these we develop a general technique for proving non-implications of $\SCFA$, $\SubPFA$ and their relatives and give several applications. For instance we show that $\SCFA$ does not imply $\MA^+(\sigma$-closed$)$ and $\SubPFA$ does not imply Martin's Maximum.
\end{abstract}

\section{Introduction}

In this paper we study variants of subcomplete and subproper forcing classes with an eye towards investigating and distinguishing their forcing principles. Subcomplete and subproper forcing are two classes of forcing notions introduced by Jensen in \cite{JensenSPSC} in connection with the extended Namba problem, see \cite[Section 6.4]{Jen14} \footnote{See Definition \ref{scspdef} below for precise definitions.}. Both are iterable with revised countable support and generalize significantly $\sigma$-closed and proper forcing notions respectively while allowing, under some circumstances, new cofinal $\omega$-sequences of ordinals to be added to uncountably cofinal cardinals. As such each comes with a forcing axiom (consistent relative to a supercompact cardinal). The forcing axiom for subcomplete forcing in particular, dubbed $\SCFA$ by Jensen in \cite{JensenCH, Jen14} is especially interesting as it is consistent with $\diamondsuit$ while implying some of the strong, structural consequences of $\MM$, see \cite[Section 4]{Jen14}. Since their initial introduction subcomplete and subproper forcing have been tied to several applications and received further treatment, see for instance, \cite{Fuchs17, Fuchs21,FuMi17,FS2020}. 

Unfortunately, there is a fly in the ointment of the birth of the theory, initially present in \cite{JensenSPSC} in the form of a missing needed assumption of $\CH$ in the proof of Lemma 1 on pg. 18. A consequence of this error led to the (false) conclusion that the $\SCFA$ implied the failure of $\square_{\omega_1}$ when in fact a careful reading of the proof of that result shows that $\SCFA$ implies the failure of $\square_{2^{\aleph_0}}$ (hence the conclusion under $\CH$), the gap was first observed by Cox. An initial starting point for us in this work was to determine if the gap was fixable and discovered that it was not. Indeed, $\SCFA$ is consistent with $\square_{\omega_1}$.

\begin{theorem}[See Theorem \ref{addasquare}]
    Assuming the consistency of a supercompact cardinal, $\SCFA$ does not imply the failure of $\square_{\aleph_1}$ when $\CH$ fails.
\end{theorem}

This result led to a general method of separating various principles related to $\SCFA$ and this method is, in essence, the subject of the present work. See \cite{FuchsERR} for a very detailed and meticulous discussion of the error as well as its propagation in the literature and corrections.


Already in \cite{FS2020} the second author and Fuchs found (seemingly) more general classes, dubbed ``$\infty$-subcomplete" and ``$\infty$-subproper" each containing their non ``$\infty$" version respectively, and proved a variety of iteration and preservation theorems. The main theorem in that work was that the forcing axiom for $\infty$-subcomplete forcing notions, $\infty$-$\SCFA$, is compatible with a large variety of behavior on $\aleph_1$ when $\CH$ fails. For instance $\aleph_1 = \mfd < \mfc = \aleph_2$ and the existence of Souslin trees are both consistent with $\infty$-$\SCFA$ $+\neg \CH$. All of these results also hold for $\SCFA$ as well (with no $\infty$).

In this paper we combine the $\infty$-versions of these forcing classes with further parametrization ``above $\mu$" for cardinals $\mu$, initially investigated, somewhat sparingly, by Jensen in \cite[Chapter 3]{variations}. This leads to a large family of forcing axioms $\infty$-$\SubPFA \hook \mu$ and $\infty$-$\SCFA \hook \mu$, where $\infty$-$\SubPFA$ and $\infty$-$\SCFA$ coincide with $\infty$-$\SubPFA \hook 2^{\aleph_0}$ and $\infty$-$\SCFA \hook 2^{\aleph_0}$ respectively. The main outcome of this work is an investigation into how these axioms relate to one another and to other, more well known axioms such as $\MM$ and $\MA^{+}(\sigma \mbox{-closed})$. Formal definitions will be given in the second part of this introduction and Section 2 but the definitions of these axioms alongside well known results provides almost immediately that the following diagram of implications holds with $2^{\aleph_0} \leq \nu < \mu$ cardinals.

\begin{figure}[h]\label{Figure.Cichon-basic}
\centering
  \begin{tikzpicture}[scale=1.5,xscale=2]
     \draw (0,4) node (MM) {$\MM$}
           
	(0,0) node (MAsigma){$\MA^+(\sigma{\rm -closed})$}
	(1,4) node (SubPFA){$\infty$-$\SubPFA$}
	(1,3) node (SubPFAnu) {$\infty$-$\SubPFA \hook \nu$}
	(1,2) node (SubPFAmu) {$\infty$-$\SubPFA \hook \mu$}
	(1,1) node (PFA) {$\PFA$}
	(1,0) node (squarePFA) {$\forall \kappa \neg \square_\kappa$}
	(2,4) node (SCFA) {$\infty$-$\SCFA$}
	(2,3) node (SCFAnu) {$\infty$-$\SCFA \hook \nu$}
	(2,2) node (SCFAmu) {$\infty$-$\SCFA \hook \mu$}
    (3,4) node (square1) {$\forall \kappa \geq 2^{\aleph_0} \neg \square_\kappa$}
    (3,3) node (square2) {$\forall \kappa \geq \nu^{\aleph_0} \neg \square_\kappa$}
    (3,2) node (square3) {$\forall \kappa \geq \mu^{\aleph_0} \neg \square_\kappa$}
           ;
     \draw[->,>=stealth]
            (MM) edge (MAsigma)
            (MAsigma) edge (squarePFA)
            (MM) edge (SubPFA)
            (SubPFA) edge (SubPFAnu)
            (SubPFA) edge (SCFA)
            (SubPFAnu) edge (SubPFAmu)
            (SubPFAnu) edge (SCFAnu)
            (SubPFAmu) edge (PFA)
            (SubPFAmu) edge (SCFAmu)
            (PFA) edge (squarePFA)
            (SCFA) edge (SCFAnu)
            (SCFA) edge (square1)
            (SCFAnu) edge (SCFAmu)
            (SCFAnu) edge (square2)
            (SCFAmu) edge (square3)
            (square1) edge (square2)
            (square2) edge (square3)
            
            ;
  \end{tikzpicture}
 \caption{Subversion forcing axioms, $\square$ principles and their relations. An arrow means direct implication.}
 \end{figure}
The main result of this work is that essentially no arrows are missing from Figure 1 above.

\begin{maintheorem}
Let $2^{\aleph_0} \leq \nu \leq \lambda < \mu = \lambda^+$ be cardinals with $\nu^\omega < \mu$. Assuming the consistency of a supercompact cardinal, the implications given in Figure 1 are complete in the sense that if no composition of arrows exists from one axiom to another then there is a model of $\ZFC$ in which the implication fails\footnote{Except for the trivial $\forall \kappa \neg \square_\kappa \to \forall \kappa \geq 2^{\aleph_0} \neg \square_\kappa$ which did not fit aesthetically into the picture.}. \label{mainthm1}
\end{maintheorem}

As a corollary of this theorem and its proof we obtain separations of several ``subversion" forcing principles from other, more well-studied reflection principles and forcing axioms. As noted above, in particular this corrects an error in the literature by showing $\SCFA$ to be consistent with $\square_{\omega_1}$. Another sample application is the following.


\begin{corollary}
    Assuming the consistency of a supercompact cardinal, $\SCFA$ does not imply $\MA^{+}(\sigma{\mbox{\rm -closed}})$.
\end{corollary}

The rest of this paper is organized as follows. In the next subsection of this introduction we give relevant background and terminology. In the next Section we introduce the variants $\infty$-subcompleteness and $\infty$-subproperness above $\mu$ and discuss some of their properties. In Section 3 we study the forcing axioms associated to these classes and show, among other things, that they are distinct as well as the fact $\infty$-$\SCFA$ implies neither $\MA^+(\sigma{\rm -closed})$ nor $\neg \square_\kappa$ for any $\kappa < 2^\omega$. In Section 4 we continue this investigation and show that $\infty$-$\SubPFA$ does not imply $\MM$. Section 5 concludes with some final remarks and open problems. 

\subsection{Preliminaries}

We conclude this introduction with the key definitions we will use throughout, beginning with that of subproperness and subcompleteness. These are these two classes of forcing notions defined by Jensen in \cite{Jen14} which have found several applications, see e.g \cite{Switzer:dissertation, JensenSPSC, Fuchs17, FuMi17}. More discussion of these concepts can be found in \cite{Jen14} or \cite{FS2020}. Before beginning with the definition we will need one preliminary definition. Below we denote by $\ZFC^-$ the axioms of $\ZFC$ without the power set axiom\footnote{There is a subtlety here, see \cite{ZFCminusPS}. As usual, we mean by $\ZFC^-$ the theory of $\ZFC$ without the powerset axiom and the replacement scheme replaced by the collection scheme, see \cite{ZFCminusPS} for full details.}. 

\begin{definition}
A transitive set $N$ (usually a model of $\ZFC^-$) is \emph{full} if there is an ordinal $\gamma$ so that $L_\gamma(N) \models \ZFC^-$ and $N$ is regular in $L_\gamma(N)$ i.e. for all $x \in N$ and $f\in L_\gamma(N)$ if $f:x \to N$ then ${\rm ran}(f) \in N$.
\end{definition}

\begin{definition}
Let $\P$ be a forcing notion and let $\delta(\P)$ be the least size of a dense subset of $\P$.

\begin{enumerate}
\item We say that $\P$ is \emph{subcomplete} if for all sufficiently large $\theta$, $\tau > \theta$ so that $H_\theta \subseteq N := L_\tau[A] \models \ZFC^-$, $s \in N$, $\sigma: \barN \prec N$ countable, transitive and full with $\sigma(\bar{\P},\bar{s}, \bar{\theta}) = \P, s, \theta$ , if $\bar{G} \subseteq \bar{\P} \cap \barN$ is generic then there is a $p \in \P$ so that if $p \in G$ is $\P$-generic over $V$ then in $V[G]$ there is a $\sigma ':\barN \prec N$ so that 

1. $\sigma ' (\bar{\P}, \bar{s}, \bar{\theta}, \bar{\mu}) = \P, s, \theta, \mu$

2. $\sigma ' ``\barG \subseteq G$

3. ${\rm Hull}^N(\delta(\P) \cup {\rm ran}(\sigma)) = {\rm Hull}^N(\delta(\P) \cup {\rm ran}(\sigma'))$
    \item 
We say that $\P$ is \emph{subproper} if for all sufficiently large $\theta$, $\tau > \theta$ so that $H_\theta \subseteq N := L_\tau[A] \models \ZFC^-$, $s \in N$, $p\in N \cap \P$, $\sigma: \barN \prec N$ countable, transitive and full with $\sigma(\bar{p}, \bar{\P},\bar{s}, \bar{\theta}) = p, \P, s, \theta$, there is a $q \in \P$ so that $q \leq p$ and if $q\in G$ is $\P$-generic over $V$ then in $V[G]$ there is a $\sigma ':\barN \prec N$ so that 

1. $\sigma ' (\bar{p}, \bar{\P}, \bar{s}, \bar{\theta}) = p, \P, s, \theta$

2. $(\sigma ' )^{-1}`` G$ is $\bar{\P}$-generic over $\bar{N}$

3. ${\rm Hull}^N(\delta(\P) \cup {\rm ran}(\sigma)) = {\rm Hull}^N(\delta(\P) \cup {\rm ran}(\sigma'))$

\end{enumerate}
\label{scspdef}
\end{definition}

Note that the special case where $\sigma = \sigma '$ is properness (for subproperness) and (up to forcing equivalence) $\sigma$-closedness (for subcomplete). To explicate this in the later case we recall the definition of {\em completeness}, which is due to Shelah originally though we take Jensen's definition\footnote{Note that Jensen defines completeness using Boolean algebras but the definition we give below can easily be seen to be equivalent.} from \cite[p. 112]{Jen14}.

\begin{definition}
    A forcing notion $\P$ is said to be {\em complete} if for all sufficiently large $\theta$ $\P \in H_\theta$ and all countable, transitive $\sigma: \barN \prec H_\theta$ with $\sigma (\bar{\P}) = \P$, if $\barG$ is $\bar{\P}$-generic over $\barN$ then there is a $p \in \P$ forcing that $\sigma `` \barG \subseteq G$.
\end{definition}

It's clear that $\sigma$-closed forcing notions are complete. What is less clear (though equally true) is that conversely if $\P$ is complete it is forcing equivalent to a $\sigma$-closed forcing notion, a result due to Jensen, see \cite[Lemma 1.3, Chapter 3]{Jen14}. In this sense therefore subcompleteness is the ``subversion" of $\sigma$-closedness.

It was pointed out in \cite{FS2020} that the ``Hulls" condition 3) in both definitions is somewhat unnatural. It is never used in applications and appears solely for the purpose of proving the iteration theorem, \cite[Theorem 3]{Jen14}. In \cite{FS2020} Fuchs and the second author showed that by iterating with Miyamoto's \emph{nice iterations} this condition could be avoided. As such it makes sense to define the following.
\begin{definition}
Let $\P$ be a forcing notion. 

\begin{enumerate}
\item We say that $\P$ is $\infty$-\emph{subcomplete} if for all sufficiently large $\theta$, $\tau > \theta$ so that $H_\theta \subseteq N := L_\tau[A] \models \ZFC^-$, $s \in N$, $\sigma: \barN \prec N$ countable, transitive and full with $\sigma(\bar{\P},\bar{s}, \bar{\theta}) = \P, s, \theta$ , if $\barG \subseteq \bar{\P} \cap \barN$ is generic then there is a $p \in \P$ so that if $p \in G$ is $\P$-generic over $V$ then in $V[G]$ there is a $\sigma ':\barN \prec N$ so that 

1. $\sigma ' (\bar{\P}, \bar{s}, \bar{\theta}, \bar{\mu}) = \P, s, \theta, \mu$

2. $\sigma ' ``\barG \subseteq G$
    \item 
We say that $\P$ is $\infty$-\emph{subproper} if for all sufficiently large $\theta$, $\tau > \theta$ so that $H_\theta \subseteq N := L_\tau[A] \models \ZFC^-$, $s \in N$, $p\in N \cap \P$, $\sigma: \barN \prec N$ countable, transitive and full with $\sigma(\bar{p}, \bar{\P},\bar{s}, \bar{\theta}) = p, \P, s, \theta$, there is a $q \in \P$ so that $q \leq p$ and if $q\in G$ is $\P$-generic over $V$ then in $V[G]$ there is a $\sigma ':\barN \prec N$ so that 

1. $\sigma ' (\bar{p}, \bar{\P}, \bar{s}, \bar{\theta}) = p, \P, s, \theta$

2. $(\sigma ' )^{-1}`` G$ is $\bar{\P}$-generic over $\bar{N}$

\end{enumerate}
\end{definition}

To be clear this is just the same as the definitions of the ``non-$\infty$" versions, simply with the additional ``Hulls" condition removed. As mentioned these classes come with an iteration theorem.

\begin{theorem}[Theorem 3.19 (for Subcomplete) and Theorem 3.20 (for Subproper) of \cite{FS2020}]
Let $\gamma$ be an ordinal and $\langle \P_\alpha, \dot{\Q}_\alpha \; | \; \alpha < \gamma\rangle$ be a nice iteration in the sense of Miyamoto so that for all $\alpha < \gamma$ we have $\forces_{\P_\alpha}$``$\dot{\Q}_\alpha$ is $\infty$-subproper (respectively $\infty$-subcomplete). Then $\P_\gamma$ is $\infty$-subproper (respectively $\infty$-subcomplete).
\end{theorem}

We note that the above theorem in the case of $\infty$-subproper forcing was originally proved first independently by Miyamoto in \cite{Miyamoto1}. A consequence of this theorem (initially observed for the non $\infty$-versions by Jensen) is that, modulo a supercompact cardinal, these classes have a consistent forcing axiom. 

\begin{definition}
Let $\Gamma$ be a class of forcing notions. The \emph{forcing axiom for} $\Gamma$, denoted $\mathsf{FA}(\Gamma)$ is the statement that for all $\P$ in $\Gamma$ and any $\omega_1$-sequence of dense subsets of $\P$, say $\{D_i\; | \; i < \omega_1\}$ there is a filter $G \subseteq \P$ which intersects every $D_i$.

If $\Gamma$ is the class of ($\infty$-)subproper forcing notions we denote $\mathsf{FA}(\Gamma)$ by ($\infty$-)$\SubPFA$. Similarly if $\Gamma$ is the class of ($\infty$-)subcomplete forcing notions we denote $\mathsf{FA}(\Gamma)$ by ($\infty$-)$\SCFA$.
\end{definition}

It is not known whether up to forcing equivalence each class is simply equal to its ``$\infty$"-version or if their corresponding forcing axioms are equivalent. However, since the ``$\infty$" versions are more general (or appear to be) and avoid the unnecessary technicality of computing hulls, we will work with them in this paper. Nearly everything written here could be formulated for the ``non-$\infty$" versions equally well, though we leave the translation to the particularly persnickety reader.

If $\Gamma \subseteq \Delta$ then $\mathsf{FA}(\Delta)$ implies $\mathsf{FA}(\Gamma)$ so we get the following collection of implications, which are part of Figure 1. 

\begin{proposition}
$\MM \to \infty \mbox{-} \SubPFA \to \PFA$ and $\MM \to \infty \mbox{-} \SubPFA \to \infty \mbox{-} \SCFA$
\end{proposition}

Here $\MM$, known as \emph{Martin's Maximum} and introduced in \cite{FMS}, is the forcing axiom for forcing notions which preserve stationary subsets of $\omega_1$ (all $\infty$-subproper forcing notions have this property) and $\PFA$ is the forcing axiom for proper forcing notions. It is known from the work of Jensen, see also \cite{FS2020} that none of the above implications can be reversed with the exception of whether $\SubPFA$ implies $\MM$. In this paper we will show the consistency of $\SubPFA + \neg \MM$, see Theorem \ref{notMM} below.

On that note we move to our last preliminary. Many of the theorems in this paper involve showing that we can preserve some fragment of $\infty$-$\SCFA$ (or $\infty$-$\SubPFA$) via a forcing killing another fragment of it. Towards this end we will need an extremely useful theorem due to Cox. Below recall that a class of forcing notions $\Gamma$ is \emph{closed under restrictions} (see Definition 39 of \cite{coxFA}) if for all $\P \in \Gamma$ and all $p \in \P$ the lower cone $\P\hook p:=\{q \in \P\; | \; q \leq p\} \in \Gamma$. One can check that both the classes of $\infty$-subcomplete and $\infty$-subproper forcing notions (as well as the restrictions ``above $\mu$" defined in Section 2) have this property.

\begin{theorem}[Cox, see Theorem 20 of \cite{coxFA}]
Let $\Gamma$ be a class of forcing notions closed under restrictions and assume $\mathsf{FA}(\Gamma)$ holds. Let $\P$ be a forcing notion. Suppose that for every $\P$-name $\dot{\Q}$ for a forcing notion in $\Gamma$ there is a $\P * \dot{\Q}$-name $\dot{\R}$ for a forcing notion so that the following hold:
\begin{enumerate}
    \item $\P * \dot{\Q} * \dot{\R}$ is in $\Gamma$,
    \item {\bf If} $j:V \to N$ is a generic elementary embedding, $\theta \geq |\P * \dot{\Q} * \dot{\R}|^+$ is regular in $V$ and
    
    a) $H_\theta^V$ is in the wellfounded part of $N$
    
    b) $j``H_\theta^V \in N$ has size $\omega_1$ in $N$
    
    c) ${\rm crit}(j) = \omega_2^V$
    
    d) There exists a $G * H * K$ in $N$ that is $(H_\theta^V, \P * \dot{\Q} * \dot{\R})$-generic
    
   \noindent {\bf Then} in $N$ the set $j``G \subseteq j(\P)$ that $j`` G$ has a lower bound in $j(\P)$ i.e. there is a $p \in j(\P) \cap N$ so that $p \leq r$ for each $r \in j``G$,
\end{enumerate}
Then $\forces_\P \mathsf{FA}(\Gamma)$ i.e. $\P$ preserves the forcing axiom for $\Gamma$. \label{Seansthm}
\end{theorem}

See \cite{coxFA} for more on strengthenings and generalizations of this wide ranging theorem. In particular a more general version stated in that article accounts for ``$+$-versions" of forcing axioms by carrying stationary sets through the list of assumptions. Since we won't use this here, we omit it. 

A typical application of Theorem \ref{Seansthm} is when $\P$ adds some object witnessing some ``non-reflective" behavior and $\mathbb R$ adds the non-reflective behavior to the full generic for $\P$ which allows $j``G$ to have a lower bound. For instance a classic result of Beaudoin, see \cite[Theorem 2.6]{Beaudoin} states that $\PFA$ is consistent with a non-reflecting stationary subset of $\omega_2$, i.e. a subset whose intersection with every point of uncountable cofinality below $\omega_2$ is not stationary. In this case the $\P$ would be the natural forcing to add such a non-reflecting set, and $\mathbb R$ would be the forcing to shoot a club through the compliment of the generic stationary set added by $\P$. The meat of Theorem \ref{Seansthm} is then that the forcing $\P$ preserves $\PFA$ if $\P * \dot{\Q} * \dot{\mathbb R}$ is proper for any proper $\dot{\mathbb Q} \in V^\P$ (which it is). A variation of this argument is made as part of Theorem \ref{notMM}, see Section 4 for details.

\section{$\infty$-Subcompleteness and $\infty$-Subproperness above $\mu$}
Most theorems in this paper filter through the notions of $\infty$-\emph{subcompleteness (respectively $\infty$-subproperness) above} $\mu$ for a cardinal $\mu$. These are technical strengthenings of $\infty$-subcompleteness (respective $\infty$-subproperness). In this section we define these strengthenings as well as make some elementary observations which will be used in rest of the paper.

\begin{definition}
Let $\mu$ be a cardinal and $\P$ a forcing notion. 
\begin{enumerate}\item
We say that $\P$ is $\infty$-\emph{subcomplete above} $\mu$ if for all sufficiently large $\theta$, $\tau > \theta$ so that $H_\theta \subseteq N := L_\tau[A] \models \ZFC^-$, $s \in N$, $\sigma: \barN \prec N$ countable, transitive and full with $\sigma(\bar{\P},\bar{s}, \bar{\theta}, \bar{\mu}) = \P, s, \theta, \mu$, if $\barG \subseteq \bar{\P} \cap \barN$ is generic then there is a $p \in \P$ so that if $p \in G$ is $\P$-generic over $V$ then in $V[G]$ there is a $\sigma ':\barN \prec N$ so that 

1. $\sigma ' (\bar{\P}, \bar{s}, \bar{\theta}, \bar{\mu}) = \P, s, \theta, \mu$

2. $\sigma ' ``\barG \subseteq G$

3. $\sigma ' \hook \bar{\mu} = \sigma \hook \bar{\mu}$

\item
We say that $\P$ is $\infty$-\emph{subproper above} $\mu$ if for all sufficiently large $\theta$, $\tau > \theta$ so that $H_\theta \subseteq N := L_\tau[A] \models \ZFC^-$, $s \in N$, $p\in N \cap \P$, $\sigma: \barN \prec N$ countable, transitive and full with $\sigma(\bar{p}, \bar{\P},\bar{s}, \bar{\theta}, \bar{\mu}) = p, \P, s, \theta, \mu$, there is a $q \in \P$ so that $q \leq p$ and if $q\in G$ is $\P$-generic over $V$ then in $V[G]$ there is a $\sigma ':\barN \prec N$ so that 

1. $\sigma ' (\bar{p}, \bar{\P}, \bar{s}, \bar{\theta}, \bar{\mu}) = p, \P, s, \theta, \mu$

2. $(\sigma ' )^{-1}`` G$ is $\bar{\P}$-generic over $\bar{N}$

3. $\sigma ' \hook \bar{\mu} = \sigma \hook \bar{\mu}$

\end{enumerate}
\end{definition}

Concretely being $\infty$-subcomplete above $\mu$ simply means that $\P$ is $\infty$-subcomplete and, moreover, for any $\sigma: \barN \prec N$ the corresponding $\sigma '$ (in $V[G]$) witnessing the $\infty$-subcompleteness can be arranged to agree with $\sigma$ ``up to $\mu$" i.e. on the ordinals below $\sigma^{-1}``\mu$ (and idem for $\infty$-subproperness). The ``non-$\infty$" versions of these classes were first introduced by Jensen in \cite{JensenSPSC} and were investigated further by Fuchs in \cite{Fuchstree} who made several of the elementary observations we repeat below. The terminology ``above $\mu$" was used by Fuchs as well as in \cite[Chapter 2]{JensenSPSC} while in other places, e.g. \cite{variations} Jensen uses the terminology ``$\mu$-subcomplete". Following the first convention, we have moved the parameter $\mu$ to the end to avoid the awkwardness of ``$\mu$-$\infty$-subcomplete/$\mu$-$\infty$-subproper". The following is immediate from the definitions.

\begin{observation}
Let $\mu < \nu$ be cardinals. If $\P$ is $\infty$-subcomplete (respectively $\infty$-subproper) above $\nu$ then it is $\infty$-subcomplete (respectively subproper) above $\mu$ and it is $\infty$-subcomplete (respectively $\infty$-subproper) (without any restriction). \label{observation1}
\end{observation}

It is easy to see that being $\infty$-subcomplete (respectively, $\infty$-subproper) is equivalent to being $\infty$-subcomplete (respectively $\infty$-subproper) above $\omega_1$, however more is true, an observation due independently to the first author and Fuchs (see \cite[Observation 4.2]{Fuchstree}, note also \cite[Observation 4.7]{Fuchstree} which is relevant here).

\begin{proposition}
Let $\P$ be a forcing notion. $\P$ is $\infty$-subcomplete (respectively $\infty$-subproper) if and only if $\P$ is $\infty$-subcomplete above $2^{\aleph_0}$ (respectively $\infty$-subproper above $2^{\aleph_0}$).
\end{proposition}

As noted above, this proposition (in the case of subcompleteness) is proved as Observation 4.2 of \cite{Fuchstree} but we give a detailed proof in order to help the reader get accustomed to $\infty$-subversion forcing as well as to include the mild difference of subproperness. However, let us note that essentially the point is that, using the definable well order in $L_\tau[A]$, the reals of $\bar{N}$ code the cardinality of the continuum. 

\begin{proof}
We prove the case of $\infty$-subproperness and leave the reader to check the case of $\infty$-subcompleteness since the latter, in its non ``$\infty$-version" can already be found in the literature. Let $\P$ be a forcing notion. It is immediate as noted above that if $\P$ is $\infty$-subproper above $2^\omega$ then it is $\infty$-subproper so we need to check just the reverse direction. Thus assume that $\P$ is $\infty$-subproper and let $\tau > \theta$ be cardinals so that $\sigma: \barN \prec N := L_\tau[A]$ with $H_\theta \subseteq N$ be as in the definition of $\infty$-subproperness. Finally let $p \in \P$ force that there is a $\sigma ' :\barN \prec N$ so that $\sigma ' (\bar{\P}) = \P$ and $\sigma '{}^{-1}G := \barG$ is $\bar{\P}$-generic over $\barN$ for any generic $G \ni p$ (the existence of such a condition is the heart of the definition of $\infty$-subproperness of course). We need to show that $p$ forces that $\sigma ' \hook 2^{\aleph_0} = \sigma \hook 2^{\aleph_0}$, where, to be clear, $2^{\aleph_0}$ denotes the cardinal (as computed in $\bar{N}$) which bijects onto the continuum (as defined in $\bar{N}$). To avoid confusion let us denote the cardinal $2^{\aleph_0} = \kappa$ (in $V$ and hence $N$) and the preimage of $\kappa$ in $\bar{N}$ under $\sigma$ as $\bar{\kappa}$. 

Fix a $G \ni p$ generic and work in $V[G]$ with $\sigma '$ etc as described in the previous paragraph. First note that by the absoluteness of $\omega$ we have that for all reals $x \in \bar{N}$ it must be the case that $\sigma (x) = \sigma' (x) = x$ (and being a real is absolute between $\bar{N}$ and $V$/$V[G]$). Moreover, since $N = L_\tau[A]$ there is a definable well order of the universe, and in particular there is a definable bijection of the reals onto $\kappa$, say $f:2^\omega \to \kappa$. By elementarity in $\bar{N}$ there is a definable bijection $\bar{f}:2^\omega \cap \bar{N} \to \bar{\kappa}$. But since $f$ is definable we have $\sigma(\bar{f}) = \sigma ' (\bar{f}) = f$ and hence for all $\alpha \in \bar{\kappa}$ we get $\sigma (\alpha) = \sigma (\bar{f}(\bar{f}^{-1}(\alpha)))= \sigma (\bar{f})(\sigma (\bar{f}^{-1}(\alpha))) = \sigma ' (\bar{f}(\bar{f}^{-1}(\alpha))) = \sigma ' (\alpha)$, as needed. Since the only assumption on $G$ was that $p \in G$ we have, back in $V$ that $p$ forces this situation which completes the proof. 
\end{proof}

Jensen showed that Namba forcing is $\infty$-subcomplete above $\omega_1$ assuming $\CH$ while it is not even $\infty$-subproper above $\omega_2$ in $\ZFC$, a consequence of the next observation, which essentially appears in \cite[Theorem 2.12]{Minden}.
\begin{lemma}
Let $\mu$ be a cardinal. 
\begin{enumerate}
    \item If $\P$ is $\infty$-subproper above $\mu$ then any new countable set of ordinals less than $\mu$ added by $\P$ is covered by an old countable set of ordinals (less than $\mu$). In particular if $\forces_\P$``${\rm cf}(\mu) = \omega$" then ${\rm cf}(\mu) = \omega$ (in $V$).
    
    \item If $\P$ is $\infty$-subcomplete above $\mu$ then $\P$ adds no new countable sets of ordinals below $\mu$.
\end{enumerate}
 \label{newseqs}
\end{lemma}

\begin{proof}
The proofs of both are similar to the corresponding proofs that every new countable set of ordinals added by a proper forcing notion is contained in an old countable set of ordinals and $\sigma$-closed forcing notions do not add new countable sets of ordinals at all respectively. The point is that to show the corresponding fact ``below $\mu$" one only needs $\infty$-subproperness (respectively $\infty$-subcompleteness) above $\mu$, see \cite[Theorem 2.12]{Minden} for details.

 \end{proof}

As mentioned before Lemma \ref{newseqs} an immediate consequence is the following.
\begin{lemma}
Namba forcing is not $\infty$-subproper above $\omega_2$. In particular Namba forcing is not $\infty$-subproper if $\CH$ fails. \label{Nambalemma}
\end{lemma}

We do not know whether this lifts to the forcing axiom level. In other words the following is open though seems unlikely given Lemma \ref{Nambalemma}.

\begin{question}
    Does $\SCFA$ imply the forcing axiom for Namba forcing when $\CH$ fails?
\end{question}

Finally we end this section with some observations about the associated forcing axioms for the classes we have been discussing.

\begin{definition}
Let $\mu$ be a cardinal. Denote by $\infty$-$\SubPFA \hook \mu$ the forcing axiom for forcing notions $\P$ which are $\infty$-subproper above $\mu$ and $\infty$-$\SCFA \hook \mu$ the same for $\P$ which are $\infty$-subcomplete above $\mu$.
\end{definition}

The following is immediate by Observation \ref{observation1}.

\begin{proposition}
Let $\mu < \nu$ be cardinals. We have that $\infty$-$\SCFA$ implies $\infty$-$\SCFA \hook \mu$ implies $\infty$-$\SCFA \hook \nu$. Similarly for the variants of $\infty$-$\SubPFA$. 
\end{proposition}
In the next section we will show that (in many cases) the reverse implications do not hold. Before doing this let us note the following which was essentially known but requires piecing together from several places in the literature (and sifting through errors given by the initial mistake detailed above).

\begin{theorem}[Essentially Jensen, \cite{JensenCH}]
    Let $2^{\aleph_0} \leq \nu \leq \kappa < \mu = \kappa^+$ be cardinals with $\nu^\omega < \mu$. The forcing axiom $\infty$-$\SCFA \hook \nu$ implies the failure of $\square_\kappa$ and even that there is no non-reflecting stationary subset of $\kappa^+ \cap {\rm cof}(\omega)$. \label{failureofsquare}
\end{theorem}

We remark that the definitions of $\square_\kappa$ and ``non-reflecting stationary set" are given in Sections 3 and 4 respectively where we use them. 

\begin{proof}
    This is essentially known though it needs to be pieced together from a few sources - particularly taking into account the error discussed before, again see \cite{FuchsERR}. First, in \cite{JensenCH}, Jensen uses the forcing notion (at $\kappa$) from \cite[Lemma 6.3 of Section 3.3]{Jen14} to obtain the failure of $\square_\kappa$ from $\SCFA$. Indeed it's easy to see that this forcing notion implies the non-existence of reflecting stationary sets and much more, see \cite{FuchsFA} for a detailed discussion of the effect of $\SCFA$ on square principles. As noted before, there is a missing assumption in the subcompleteness of the relevant forcing - namely that $\kappa > 2^{\aleph_0}$. Second, \cite[Lemma 3.5]{Fuchs21}, which contains no errors as written, implies that the forcing notion needed is indeed $\infty$-subcomplete above $\nu$ under the cardinal arithmetic assumptions mentioned in the theorem statement. See the proof of \cite[Lemma 3.5]{Fuchs21} and the discussion therein for more details.
\end{proof}

\section{Separating the $\infty$-$\SCFA \hook \mu$ Principles}

In this section we show that under certain cardinal arithmetic assumptions $\infty$-$\SCFA \hook \nu$ does not imply $\infty$-$\SCFA \hook \mu$ for $\mu < \nu$. Before proving this general theorem we introduce our technique with the simple example of separating $\infty$-$\SCFA \hook \omega_1$ from $\infty$-$\SCFA \hook \omega_2$. This involves showing that adding a $\square_{\omega_1}$-sequence to a model of $\infty$-$\SCFA$ preserves $\infty$-$\SCFA \hook \omega_2$. By contrast note that Theorem \ref{failureofsquare} proves that $\SCFA + \CH$ implies the failure of $\square_{\omega_1}$. Let us remark one more time that, as stated in the introduction the fact that $\SCFA$ can coexist with a $\square_{\omega_1}$-sequence closes the door on the aforementioned error by showing that the argument cannot be resurrected when $\CH$ fails. 

This case is treated as a warm-up and we extract from it a more general lemma for preservation of axioms of the form $\infty$-$\SCFA \hook \mu^+$ from which the other separation results are then derived.

\subsection{The Case of $\infty$-$\SCFA \hook \omega_2$: Adding a $\square_{\omega_1}$ Sequence}

Recall that for an uncountable cardinal $\lambda$ a $\square_\lambda$-sequence is a sequence $\langle C_\alpha \; | \; \alpha \in \lambda^+ \cap {\rm Lim}\rangle$ so that for all $\alpha$ the following hold:

\begin{enumerate}
    \item $C_\alpha$ is club in $\alpha$
    \item ${\rm ot}(\alpha) \leq \lambda$
    \item For each $\beta \in {\rm lim}(C_\alpha)$ we have that $C_\alpha \cap \beta = C_\beta$
\end{enumerate}

We recall the poset $\P_0$ from \cite[Example 6.6]{CummingsHB} for adding a square sequence. Conditions $p \in \P_0$ are functions so that the domain of $p$ is $\beta + 1 \cap {\rm Lim}$ for some $\beta \in \lambda^+ \cap {\rm Lim}$ and

\begin{enumerate}
    \item For all $\alpha \in {\rm dom}(p)$ we have that $p(\alpha)$ is club in $\alpha$ with order type $\leq \lambda$; and
    \item If $\alpha \in {\rm dom}(p)$ then for each $\beta \in {\rm lim} (p(\alpha))$ we have $p(\alpha) \cap \beta = p(\beta)$
\end{enumerate}
The order is end extension. We remark that a moment's reflection confirms that this poset is $\sigma$-closed. Moreover it is ${<}\lambda^+$-strategically closed (see \cite{CummingsHB}). In particular it preserves cardinals up to $\lambda^+$.

\begin{theorem}
Assume $\infty$-$\SCFA \hook \omega_2$ and let $\P_0$ be the forcing notion defined above for adding a $\square_{\omega_1}$-sequence. Then $\forces_{\P_0}$ $\infty$-$\SCFA \hook \omega_2$. In particular if the existence of a supercompact cardinal is consistent with $\ZFC$ then $\infty$-$\SCFA \hook \omega_2 + \square_{\omega_1}$ is consistent as well.\label{addasquare}
\end{theorem}

Before proving this theorem we need to define one more poset. Recall that if $G \subseteq \P_0$ is generic and $\vec{\mathcal C}_G = \langle C_\alpha \; | \; \alpha \in \lambda^+ \cap {\rm Lim}\rangle$ is the generic $\square_\lambda$-sequence added by $G$ then for any cardinal $\gamma < \lambda$ we can \emph{thread the square sequence} via the following poset, $\T_{G, \gamma}$. Conditions are closed, bounded subsets $c \subseteq \lambda^+$ so that $c$ has order type $<\gamma$, and for all limit points $\beta \in c$ we have that $\beta \cap c = C_\beta$. See \cite[\S 6]{cummings2001squares} and \cite[p.7]{MLH} for more on this threading poset. The point is the following.
\begin{fact}[Lemma 6.9 of \cite{cummings2001squares}]
Let $\gamma < \lambda$ be cardinals, $\P_0$ the forcing notion described above for adding a $\square_\lambda$-sequence and $\dot{\T}_{\dot{G}, \gamma}$ be the $\P_0$-name for the forcing to thread the generic square sequence with conditions of size $<\gamma$. Then $\P_0 * \dot{\T}_{\dot{G}, \gamma}$ has a dense $<\gamma$-closed subset.
\end{fact}

We can now prove Theorem \ref{addasquare}.

\begin{proof}
We let $\P_0$ be the forcing described above for adding a $\square_{\omega_1}$-sequence (so $\lambda = \omega_1$). Let $\gamma = \aleph_1$ so in $V^{\P_0}$ the threading poset $\dot{\T}:=\dot{\T}_{\dot{G}, \aleph_1}$ consists of countable closed subsets of $\omega_2$. We want to apply Theorem \ref{Seansthm} to $\P_0$. Note that if $\dot{\Q}$ is a $\P_0$-name for an $\infty$-subcomplete above $\omega_2$ forcing notion, then $\dot{\T} =\dot{\T}_{\dot{G}, \aleph_1}$ is absolute between $V^{\P_0}$ and $V^{\P_0 * \dot{\Q}}$ by Lemma \ref{newseqs} (2).

\begin{claim}
It is enough to show that for any $\P_0$-name $\dot{\Q}$ for a forcing notion which is $\infty$-subcomplete above $\omega_2$, the three step $\P_0 * \dot{\Q} * \dot{\T}$ is $\infty$-subcomplete above $\omega_2$. \label{thread_lowerbound}
\end{claim} 

\begin{proof}[Proof of Claim]
This is because $\T$ adds a lower bound to $j``G$ as described in the statement of Theorem \ref{Seansthm}. In more detail, let $\dot{\Q}$ be a $\P_0$-name for a forcing notion which is $\infty$-subcomplete above $\omega_2$, we want to show that for $\dot{\R} = \dot{\T}$ the hypotheses of Theorem \ref{Seansthm} are satisfied assuming that $\P_0 * \dot{\Q} * \dot{\T}$ is $\infty$-subcomplete above $\omega_2$. Since this is exactly the first clause we only need to concern ourselves with the second one. Recall that, relativized to this situation, this says that {\bf if} $j:V \to N$ is a generic elementary embedding, $\theta \geq |\P_0 * \dot{\Q} * \dot{\T}|^+$ is regular in $V$ and
    
    a) $H_\theta^V$ is in the wellfounded part of $N$;
    
    b) $j``H_\theta^V \in N$ has size $\omega_1$ in $N$;
    
    c) ${\rm crit}(j) = \omega_2^V$
    
    d) There exists a $G * H * K$ in $N$ that is $(H_\theta^V, \P_0 * \dot{\Q} * \dot{\T})$-generic
    
   \noindent {\bf Then} $N$ believes that $j`` G$ has a lower bound in $j(\P_0)$.
   
So fix some $\theta$ and $j:V \to N$ as described in a) to d). Note that $j`` G = G$ by c) and the fact that $G$ is coded as a subset of $\omega_2^V$. Thus it suffices to find a lower bound of $G$ in $j(\P_0)$. The point is now though that since $G * H * K \in N$ we can in particular form $\bigcup K \in N$ which is a club subset of $\omega_2^V = \sup_{p \in G} {\rm dom} (p)$ and coheres with all of the elements of $G$, and hence $(\bigcup G) \cup \langle \omega_2^V , \bigcup K \rangle$ is as needed.
\end{proof}

Let us now show that $\P_0 * \dot{\Q} * \dot{\T}$ is $\infty$-subcomplete above $\omega_2$. Let $\tau > \theta$ be sufficiently large cardinals and $\sigma:\barN \prec N = L_\tau[A] \supseteq H_\theta$ be as in the definition of $\infty$-subcompleteness above $\omega_2$. Let $\sigma (\bar{\P}_0, \dot{\bar{\Q}}, \dot{\bar{\T}}) = \P_0, \dot{\Q}, \dot{\T}$. Let $\barG * \bar{H} * \bar{K}$ be $\bar{\P}_0 * \dot{\bar{\Q}} * \dot{\bar{\T}}$-generic over $\barN$. There are few things to note. First let us point out that $\barG$ and $\bar{K}$ are (coded as) subsets of $\bar{\omega}_2$, the second uncountable cardinal from the point of view of $\barN$ (so $\sigma(\bar{\omega}_2) = \omega_2$). Next note that $\P_0 * \dot{\Q} * \dot{\T}$ is isomorphic to $\P_0 * \dot{\T} * \dot{{\Q}}$ since both $\dot{\Q}$ and $\dot{\T}$ are in $V^{\P_0}$, and the same for the ``bar" versions in $\barN$ (i.e. we have a product not an iteration for the second and third iterands). Now note that since $\P_0 *\dot{\T}$ has a $\sigma$-closed dense subset, $\sigma `` \barG * \bar{K}$ has a lower bound (in $N$), say $(p, t)$ ($t$ is in the ground model and the $\sigma$-closed dense subset is simply the collection of conditions whose second coordinate is a check name decided by $p$). By $\sigma$-closedness (which again is implies completeness) $(p, t)$ forces that there is a unique lift of $\sigma:\bar{N} \prec N$ to some $\sigma_0:\barN[\barG] \prec N[G]$ with $\sigma_0(\barG) = G$ for any $\P_0$-generic $G\ni p$ (technically we need to work in the extension by $\P_0 * \dot{\T}$, but we only want to specify the embedding of the $\bar{\P}_0$ extension). Fix such a $G$ (from which $\sigma_0$ is defined) and work in $V[G]$. Note that $\sigma_0 ``\bar{K} = \sigma ``\bar{K}$ has $t \in N$ as a lower bound. Now in $V[G]$ (NOT $V[G][K]$) we have that $\Q := \dot{\Q}^G$ is $\infty$-subcomplete above $\omega_2$ as $\P$ forced this to be so by assumption. Therefore in $V[G]$ we can apply the definition of $\infty$-subcompleteness to $\sigma_0: \barN[\barG] \prec N[G]$ to obtain a condition $\dot{q}^G:= q \in \Q$ so that if $H \ni q$ is $\Q$-generic over $V[G]$ then in $V[G][H]$ there is a $\sigma_1:\bar{N}[\barG] \prec N[G]$ so that $\sigma_1(\barG, \bar{\P}_0, \dot{\bar{\Q}}^{\barG}, \dot{\bar{\T}}^{\barG}) = G, \P_0, \Q, \T$ where $\T \in V[G]$ is $\dot{\T}^G$, $\sigma_1 `` \bar{H} \subseteq H$ and $\sigma_1 \hook \bar{\omega}_2 = \sigma \hook \bar{\omega}_2$. Note also that by condensation we have that $\barN = L_{\bar{\tau}}[\bar{A}]$ and hence we can ensure that $\sigma_1 \hook \barN: \barN \prec N$. Let us denote by $\sigma_2$ this restriction $\sigma_1 \hook \barN$. As this is an element of $V[G][H]$ there is, in $V$ a $\P_0 * \dot{\Q}$-name for this embedding, which we will call $\dot{\sigma_2}$. 

Now by the first observation above we know that since $\barG$ and $\bar{K}$ are coded as subsets of $\bar{\omega}_2$ so it must be the case that in fact $\sigma_2 \hook \barG = \sigma \hook \barG$ and idem for $\bar{K}$. In particular $(p, t)$ is still a lower bound of $\sigma_1 `` \barG * \bar{K}$. But putting all of these observations together now ensures that the triple $(p, \dot{q}, t) \in \P_0 * \dot{\Q} * \dot{\T}$ forces that $\sigma_2 := \sigma_1 \hook \barN$ is as needed to witness that the three step is $\infty$-subcomplete above $\omega_2$ as needed.
\end{proof}

Note the following corollary of Theorem \ref{addasquare}.
\begin{corollary}
The forcing axiom $\infty$-$\SCFA$ does not imply $\MA^+(\sigma{\mbox{\rm -closed}})$ assuming the consistency of a supercompact cardinal. In particular $\infty$-$\SCFA$ does not imply $\SCFA^+$. \label{separate1}
\end{corollary}

\begin{proof}
Begin with a model of $\infty$-$\SCFA + 2^{\aleph_0} = 2^{\aleph_1} = \aleph_2$ (for instance a model of $\MM$). Force with $\P_0$ to preserve these axioms and add a $\square_{\omega_1}$-sequence. Then $\infty$-$\SCFA \hook \omega_2$ and $\square_{\omega_1}$ hold in the extension by Theorem \ref{addasquare}. But, since $\P_0$ does not collapse cardinals (by $2^{\aleph_1} = \aleph_2$) or add reals, the continuum is still $\aleph_2$ hence $\infty$-$\SCFA$ holds yet $\MA^+(\sigma \mbox{-closed})$ fails since this axiom implies that $\square_\kappa$ fails for all $\kappa$, see \cite{FMS}.
\end{proof}

\subsection{The General Case}
The proof of Theorem \ref{addasquare} can be generalized in many ways. Observe that very little about $\P_0$ and $\dot{\mathbb T}$ were used. In fact essentially the same proof as above really shows the following general metatheorem. 

\begin{theorem}
    Let $\mu$ be an uncountable cardinal . Let $\P$ be a poset whose conditions as well as any generic $G$ can be coded by subsets of $\mu^+$ and let $\dot{\R}$ be a $\P$-name for a poset which is forced to be so that all of its conditions and any generic $K$ are coded by subsets of $\mu^+$. Assume moreover that $\P *\dot{\R}$ has a $\sigma$-closed dense subset and $\forces_\P \dot{\R} \subseteq V$ i.e. all of the elements of $\dot{R}$ are in the ground model\footnote{For instance in the case of Theorem \ref{addasquare} this follows from the strategic closure.}. Then for every $\dot{\Q}$ a $\P$-name for a $\infty$-subcomplete above $\mu^+$ poset the poset $\P * \dot{\Q} * \dot{\R}$ is $\infty$-subcomplete above $\mu^+$. 
    
    Consequently if $\P * \dot{\Q} * \dot{\R}$ satisfies (2) of Theorem \ref{Seansthm} and $\infty$-$\SCFA \hook \mu^+$ holds then $\P$ preserves $\infty$-$\SCFA \hook \mu^+$. \label{meta}
\end{theorem}

\begin{proof}
    This is really just an abstraction of what we have already seen. Let $\tau > \theta$ be sufficiently large cardinals and $\sigma:\barN \prec N = L_\tau[A] \supseteq H_\theta$ be as in the definition of $\infty$-subcompleteness above $\mu^+$. Let $\sigma (\bar{\P}, \dot{\bar{\Q}}, \dot{\bar{\R}}, \bar{\mu}) = \P, \dot{\Q}, \dot{\R}, \mu$. Let $\barG * \bar{H} * \bar{K}$ be $\bar{\P} * \dot{\bar{\Q}} * \dot{\bar{\R}}$-generic over $\barN$. As in Theorem \ref{addasquare} note that first of all $\barG$ and $\bar{K}$ are (coded as) subsets of $\bar{\mu}^+$ (note $\bar{\mu}^+$, the successor of $\bar{\mu}$ as computed in $\barN$ is the same as $\bar{\mu^+}$, the preimage of $\mu^+$ under $\sigma$ by elementarity). Next note that $\P * \dot{\Q} * \dot{\R}$ is isomorphic to $\P * \dot{\R} * \dot{{\Q}}$ since both $\dot{\Q}$ and $\dot{\R}$ are in $V^{\P}$, and the same for the ``bar" versions in $\barN$ (i.e. we have a product not an iteration for the second and third iterands), just as before. Now note that since $\P *\dot{\R}$ has a $\sigma$-closed dense subset, there is a condition $(p, t) \in \P * \dot{\R}$ forcing $\sigma `` \barG * \bar{K}$ to be contained in the generic and moreover this condition is a lower bound on $\sigma ``\barG * \bar{K}$ by elementarity: since $\bar{N}$ thinks $\bar{\P} * \dot{\bar{\R}}$ has a $\sigma$-closed dense subset densely many of the conditions in $\barG * \bar{K}$ are in this set and hence their images are in the real $\sigma$-closed dense subset of $\P * \dot{\R}$ which in turn implies that we can find the lower bound. Note this condition $(p, t)$ is in $N$ and by the assumption that $\dot{\R}$ is forced to be contained in the ground model we can assume that $t \in V$ (and in fact in $N$). It follows that $(p, t)$ forces that there is a unique lift of $\sigma:\bar{N} \prec N$ to some $\sigma_0:\barN[\barG] \prec N[G]$ with $\sigma_0(\barG) = G$ for any $\P$-generic $G\ni p$ (technically we need to work in the extension by $\P * \dot{\R}$, but we only want to specify the embedding of the $\bar{\P}$ extension). Fix such a $G$ (from which $\sigma_0$ is defined) and work in $V[G]$. Note that $\sigma_0 ``\bar{K} = \sigma ``\bar{K}$ and, as already stated, $t \in N$ is a lower bound. Since $\Q := \dot{\Q}^G$ was forced to be $\infty$-subcomplete above $\mu^+$, working in $V[G]$ we can apply the definition of $\infty$-subcompleteness to $\sigma_0: \barN[\barG] \prec N[G]$ to obtain a condition $\dot{q}^G:= q \in \Q$ so that if $H \ni q$ is $\Q$-generic over $V[G]$ then in $V[G][H]$ there is a $\sigma_1:\bar{N}[\barG] \prec N[G]$ so that $\sigma_1(\barG, \bar{\P}, \dot{\bar{\Q}}^{\barG}, \dot{\bar{\R}}^{\barG}) = G, \P, \Q, \R$ where $\R \in V[G]$ is $\dot{\R}^G$, $\sigma_1 `` \bar{H} \subseteq H$ and $\sigma_1 \hook \bar{\mu}^+ = \sigma \hook \bar{\mu}^+$. Note also that by condensation we have that $\barN = L_{\bar{\tau}}[\bar{A}]$ and hence we can ensure that $\sigma_1 \hook \barN: \barN \prec N$. Let us denote by $\sigma_2$ the embedding $\sigma_1 \hook \barN$ and let $\dot{\sigma_2}$ be a $\P * \dot{\Q}$-name for $\sigma_2$ in $V$.

Now $\barG$ and $\bar{K}$ are coded as subsets of $\bar{\mu}^+$ by assumption. Therefore it must be the case that in fact $\sigma_1 \hook \barG = \sigma \hook \barG$ and idem for $\bar{K}$ - note the subtlety here $\bar{K}$ is not in $\barN[\bar{G}]$ but is a subset of it. In particular $(p, t)$ is still a lower bound in $\sigma_1 `` \barG * \bar{K}$. But putting all of these observations together now ensures that the triple $(p, \dot{q}, t) \in \P * \dot{\Q} * \dot{\T}$ forces that $\dot{\sigma}_2$ is as needed to witness that the three step is $\infty$-subcomplete above $\mu^+$ as needed.
\end{proof}

Before moving to our main application let us give another one at the level of $\omega_2$.

\begin{theorem}
Assume $\infty$-$\SCFA \hook \omega_2$. The forcing $\mathbb{S}_{\omega_2}$ to add an $\omega_2$-Souslin tree preserves $\infty$-$\SCFA \hook \omega_2$. 
\end{theorem}

\begin{proof}[Sketch]
Let $\mathbb{S}_{\omega_2}$ be the standard forcing to add an $\omega_2$-Souslin tree: conditions are binary trees $p \subseteq 2^{<\omega_2}$ of size $< \aleph_2$ ordered by end extension. This adds an $\omega_2$-Souslin tree and is $\sigma$-closed. Let $\dot{T}_{\dot{G}}$ be the canonical name for the tree added i.e. if $G\subseteq \mathbb{S}_{\omega_2}$ is generic over $V$ then $(\dot{T}_{\dot{G}})^G = \bigcup G$. Let $\dot{\Q}$ be a $\mathbb{S}_{\omega_2}$-name for a forcing notion which is $\infty$-subcomplete above $\omega_2$. As before it is enough to show that $\mathbb{S}_{\omega_2} * \dot{\Q} * \dot{T}_{\dot{G}}$ is $\infty$-subcomplete above $\omega_2$ where $\dot{T}_{\dot{G}}$ is the name for the tree as a forcing notion, by essentially the same proof as in the case of Theorem \ref{addasquare}. However that this three step is $\infty$-subcomplete above $\omega_2$ now follows almost immediately from Theorem \ref{meta}. 
\end{proof}

We have the following corollary similar to Corollary \ref{separate1} above by invoking a model of $\infty$-$\SCFA + 2^{\aleph_0} = \aleph_2$.

\begin{corollary}
Assuming the consistency of a supercompact cardinal we have the consistency of $\SCFA + \neg \CH + \neg {\rm TP}(\omega_2)$.
\end{corollary}
Here ${\rm TP}(\omega_2)$ is the tree property at $\omega_2$ i.e. no $\omega_2$-Aronszajn trees. This result contrasts with \cite[Corollary 4.1]{TPW} which shows that under Rado's Conjecture, another forcing axiom-like statement compatible with $\CH$, ${\rm TP}(\omega_2)$ is equivalent to $\neg \CH$. 

The proof of Theorem \ref{addasquare}, using Theorem \ref{meta} can be easily generalized to establish that for any cardinal $\mu$ adding a $\square_\mu$ sequence via $\P_0$ preserves $\infty$-$\SCFA\hook \mu^+$.

\begin{theorem}
Let $\mu$ be an uncountable cardinal and assume $\infty$-$\SCFA \hook \mu^+$ holds. If $\P_0$ is the forcing from the previous subsection to add a $\square_{\mu}$-sequence then $\P_0$ preserves $\infty$-$\SCFA \hook \mu^+$. \label{addabiggersquare}
\end{theorem}

\begin{proof}
In $V^{\P_0}$ let $\dot{\T} := \dot{\T}_{\dot{G},\aleph_1}$.
We only give the proof of the claim obtained from Claim \ref{thread_lowerbound} by replacing $\omega_2$ with $\mu$. The other part of the proof - that the requisite three step forcing is $\infty$-subcomplete above $\mu^+$ is an immediate consequence of Theorem \ref{meta}.

Suppose $j : V \to N$, $\theta$ and $G * H * K$ are as in the proof of Claim \ref{thread_lowerbound}. Let $\beta := (\mu^+)^V = \sup_{p \in G} {\rm dom} (p)$. Then $\bigcup K \in N$ is a club subset of $\beta$ and coheres with all of the elements of $G$. Note that all initial segments of $\bigcup K$ are countable sets in $V$. So $K^* := j `` \bigcup K$ is club in $\beta^* := \sup (j `` \beta)$ and coheres with all of the elements of $G^* := j `` G$. Hence $(\bigcup G^*) \cup \langle \beta^* , K^* \rangle$ is a lower bound of $j `` G$ in $j(\P_0)$.
\end{proof}


Putting all of these results together we get the following.

\begin{theorem}
Let $2^{\aleph_0} \leq \nu \leq \kappa < \mu = \kappa^+$ be cardinals with $\nu^\omega < \mu$. Modulo the existence of a supercompact cardinal $\infty \mbox{-} \SCFA \hook \mu + \neg \infty \mbox{-} \SCFA \hook \nu$ is consistent.
\end{theorem}

\begin{proof}
By Theorem \ref{addabiggersquare} we know that $\infty$-$\SCFA \hook \mu$ is consistent with $\square_\kappa$ hence it suffices to see that $\infty$-$\SCFA \hook \nu$ implies the failure of $\square_\kappa$, but this is exactly the content of Theorem \ref{failureofsquare} above. 
\end{proof}

\section{Separating $\MM$ from $\SubPFA$}

In this section we prove the following result.

\begin{theorem}
    Assume there is a supercompact cardinal. Then there is a forcing extension in which $\infty$-$\SubPFA$ holds but $\MM$ fails. In particular, modulo the large cardinal assumption, $\infty$-$\SubPFA$ does not imply $\MM$. \label{notMM}
\end{theorem}

The idea behind this theorem is a combination of the proof technique from \cite[Theorem 2.6]{Beaudoin} and the proof of Theorem \ref{addasquare}. Starting from a model of $\MM$ we will force to add a non-reflecting stationary set to $2^{\aleph_0}$ ($=\aleph_2$ since $\MM$ holds). This kills $\MM$ by the results of \cite{FMS} but will preserve $\infty$-$\SubPFA$ by an argument similar to that of \cite[Theorem 2.6]{Beaudoin}. In that paper Beaudoin proves that in fact $\PFA$ is consistent with a nonreflecting stationary subset of any regular cardinal $\kappa$. The interesting difference in the subproper case is that $\infty$-$\SubPFA$ (in fact $\SCFA$) implies that there are no non-reflecting stationary subsets of any cardinal greater than the size of the continuum, see Theorem \ref{failureofsquare} above. In short, $\PFA$ is consistent with a nonreflecting stationary subset of every regular cardinal $\kappa$ while $\infty$-$\SubPFA$ is only consistent with a nonreflecting stationary subset of $\omega_2$. We begin by recalling the relevant definitions.

\begin{definition}
    Let $\kappa$ be a cardinal of uncountable cofinality and $S \subseteq \kappa$. For a limit ordinal $\alpha < \kappa$ of uncountable cofinality we say that $S$ {\em reflects to} $\alpha$ if $S \cap \alpha$ is stationary in $\alpha$. We say that $S$ is {\em non-reflecting} if it does not reflect to any $\alpha < \kappa$ of uncountable cofinality. 
\end{definition}

\begin{fact}[See Theorem 9 \cite{FMS}]
    $\MM$ implies that for every regular $\kappa > \aleph_1$ every stationary subset of $\kappa \cap {\rm Cof}(\omega)$ reflects.\label{MMstat}
\end{fact} 

Compare this with the following, which was also noted in the proof of Theorem \ref{failureofsquare} above. 

\begin{fact}[See Lemma 6, Section 4 of \cite{Jen14}]
    $\SCFA$ implies that for every regular $\kappa > 2^{\aleph_0}$ every stationary subset of $\kappa \cap {\rm Cof}(\omega)$ reflects.
\end{fact}

\begin{remark}
Again, in \cite{Jen14} it is claimed that $\SCFA$ implies that the above holds for all $\kappa > \aleph_1$, regardless of the size of the continuum. However, this too is incorrect without $\CH$ because of the error. 
\end{remark}

There is a natural forcing notion to add a non-reflecting stationary subset $S \subseteq \kappa \cap {\rm Cof}(\omega)$ for a fixed regular cardinal $\kappa$. The definition and basic properties are given in Example 6.5 of \cite{CummingsHB}. We record the basics here for reference.
\begin{definition}
Fix a regular cardinal $\kappa > \aleph_1$. The forcing notion $\mathbb{NR}_\kappa$ is defined as follows. Conditions are functions $p$ with domain the set of countably cofinal ordinals below some ordinal $\alpha < \kappa$ mapping into $2$ with the property that if $\beta \leq {\rm sup}({\rm dom}(p))$ has uncountable cofinality then there is a set $c \subseteq \beta$ club in $\beta$ which is disjoint from $p^{-1} (1) = \{ \alpha \in {\rm dom} (p) \mid p(\alpha) = 1 \}$. The extension relation is simply $q \leq_{\mathbb{NR}_\kappa} p$ if and only if $q \supseteq p$. 
\end{definition}

Proofs of the following can be found in \cite{CummingsHB}. 
\begin{proposition}
For any regular $\kappa > \aleph_1$ the forcing $\mathbb{NR}_\kappa$ has the following properties. 
\begin{enumerate}
    \item $\mathbb{NR}_\kappa$ is $\sigma$-closed.
    \item $\mathbb{NR}_\kappa$ is $\kappa$-strategically closed and in particular preserves cardinals.
    \item If $G\subseteq \mathbb{NR}_{\kappa}$ is generic then $S_G := \bigcup_{p \in G} p^{-1} (1)$ is a non-reflecting stationary subset of $\kappa$.
\end{enumerate}
\end{proposition}
We neglect to give the definition of strategic closure since we will not need it beyond the fact stated above, see \cite{cummings2001squares} or \cite{CummingsHB} for a definition. 

Let $\kappa$ be as above, $G\subseteq \mathbb{NR}_\kappa$ be generic over $V$ and let $S_G := \bigcup_{p \in G} p^{-1} (1)$ be the generic non-reflecting stationary set. We want to define a forcing to kill $S_G$ (this will be the ``$\dot{\R}$" in our application of Theorem \ref{Seansthm}). Specifically we will define a forcing notion $\Q_{S_G}$ so that forcing with $\Q_{S_G}$ will add a club to $\kappa \setminus S_G$ and hence kill the stationarity of $S_G$. Note that since $S_G$ is non-reflecting its complement must also be stationary and indeed has to be {\em fat}, i.e. contain continuous sequences of arbitrary length $\alpha < \kappa$ cofinally high.

\begin{definition}
    Borrowing the notation from the previous paragraph define the forcing notion $\Q_{S_G}$ as the set of closed, bounded subsets of $\kappa \setminus S_G$ ordered by end extension.
\end{definition}
Clearly the above forcing generically adds a club to the complement of $S_G$ thus killing its stationarity, see \cite{CummingsHB} Definition 6.10. It is also $\omega$-distributive. 

We are now ready to prove Theorem \ref{notMM}.

\begin{proof}[Proof of Theorem \ref{notMM}]
    Assume $\infty$-$\SubPFA$ holds (the consistency of this is the only application of the supercompact). Note that the continuum is $\aleph_2$ and will remain so in any cardinal preserving forcing extension which adds no reals. Let $\P = \mathbb{NR}_{\aleph_2}$, $G \subseteq \P$ be generic over $V$ and work in $V[G]$. Obviously in this model we have ``there is a non-reflecting stationary subset of $\aleph_2$" and thus $\MM$ fails by Fact \ref{MMstat}. We need to show that $\infty$-$\SubPFA$ holds. 

    We will apply Theorem \ref{Seansthm} much as in the proof of Theorem \ref{addasquare}. Let $\dot{\Q}$ be a $\P$-name for an $\infty$-subproper forcing notion and let $\dot{\R}$ name $\Q_{S_{\dot{G}}}$ in $V^{\P * \dot{\Q}}$ (NOT just in  $V^\P$ - this is different than the proof of Theorem \ref{addasquare} and crucial). By exactly the same argument as in the proof of Theorem \ref{addasquare} it suffices to show that $\P * \dot{\Q} * \dot{\R}$ is $\infty$-subproper (in $V$). This is because (2) from Theorem \ref{Seansthm} follows from the fact that, borrowing the notation from the statement of that theorem applied to our situation $\dot{\R}$ shoots a club through the complement of $S_G$ hence $j`` S_G = S_G$ is non-stationary in its supremum and so has a lower bound in $N$. 
    
    So we show that $\P * \dot{\Q} * \dot{\R}$ is $\infty$-subproper. This is very similar to the proof of Theorem \ref{addasquare} or even Theorem \ref{meta} more generally but enough details are different to warrant repeating everything for completeness. Let $\tau > \theta$ be sufficiently large cardinals and $\sigma:\barN \prec N = L_\tau[A] \supseteq H_\theta$ be as in the definition of $\infty$-subproperness. Let $\sigma (\bar{\P}, \dot{\bar{\Q}}, \dot{\bar{\R}}, \bar{\omega_2}) = \P, \dot{\Q}, \dot{\R}, \omega_2$. Let $(p_0, \dot{q}_0, \dot{r}_0)$ be a condition in $\P * \dot{\Q} * \dot{\R}$ with $\sigma(\bar{p}_0, \dot{\bar{q}}_0, \dot{\bar{r}}_0) = (p_0, \dot{q}_0, \dot{r}_0)$. Applying the $\sigma$-closure of $\P$ we can find a $\bar{\P}$-generic $\barG$ over $\barN$ and a condition $p \leq p_0$ so that $p$ is a lower bound on $\sigma `` \bar{G}$ and, letting $\alpha = {\rm sup}(\sigma ``\bar{\omega}_2)$, we have $p(\alpha) = 0$ (i.e. $p$ forces $\alpha$ to not be in the generic stationary set). Let us assume $p \in G$ and note that this condition forces $\sigma `` \barG \subseteq G$ and hence $\sigma$ lifts uniquely to a $\Tilde{\sigma}:\barN[\barG] \prec N[G]$ that $\Tilde{\sigma} (\barG) = G$ and $\alpha :={\rm sup}(\sigma ``\bar{\omega}_2) \notin S_G$. Let $\bar{\Q} = \dot{\bar{\Q}}^{\barG}$ as computed in $\barN[\barG]$ and let $\bar{q}_0 = \dot{\bar{q}}_0^{\bar{G}} \in \barN[\barG]$. Applying the fact that $\dot{\Q}$ is forced to be $\infty$-subproper let $q \leq q_0 = \Tilde{\sigma}(\bar{q}_0)$ be a condition forcing that if $H \subseteq \Q$ is $V$-generic with $q \in H$ then there is a $\sigma ' \in V[G][H]$ so that $\sigma':\bar{N}[\barG] \prec N[G]$ as in the definition of $\infty$-subproperness (with respect to $\Tilde{\sigma}$). Note that as in the proof of Theorem \ref{addasquare} $\sigma '\hook \barN:\barN \prec N$ and $\sigma ' \hook \bar{\omega}_2 = \sigma \hook \bar{\omega}_2$. Let $\Tilde{\sigma}' : \bar{N}[\bar{G}][\bar{H}] \to N[G][H]$ be the lift of $\sigma '$, where $\bar{H} = ( \sigma ' )^{-1} `` H$.

    \begin{claim}
    In $V[G][H]$ the set $S_G$ does not contain a club.
    \end{claim}

    \begin{proof}[Proof of Claim]
        Since $\aleph_2$ is the continuum in $V[G]$ note that $\omega_2^{V[G]}$ remains uncountably cofinal in $V[G][H]$ (though of course it can be collapsed to $\omega_1$). Suppose towards a contradiction that $S_G$ contains a club and note that since we chose $\theta$ and $\tau$ to be sufficiently large with respect to the forcing (and therefore in particular we can assume $H_\theta$ contains the powerset of $\omega_2$) we have $N[G][H] \models$ ``$\exists C$ which is club and $C \subseteq S_G$". By elementarity there is a $\bar{C} \in \barN[\barG][\bar{H}]$ so that $$\barN[\barG][\bar{H}] \models \bar{C} \subseteq \bar{S}_G \; {\rm is\, club}$$ where $\bar{H} :=\sigma '{}^{-1}H$ is $\bar{\Q}$-generic over $\barN[\barG]$ by the definition of $\infty$-subcompleteness and the choice of $q$. But now note that if $C = \Tilde{\sigma}'(\bar{C})$ then $C \cap \alpha$ is cofinal in $\alpha$ by elementarity so $\alpha \in C$ but $\alpha \notin S_G$ which is a contradiction.
    \end{proof}
    
    Given the claim we know that $\omega_2^V \setminus S_G$ is a stationary set in $V[G][H]$ and hence $\R := \dot{\R}^{G * H}$ is the forcing to shoot a club through a stationary set. Let $\bar{\R} \in \barN[\barG][\bar{H}]$ be $\dot{\bar{\R}}^{\barG * \bar{H}}$. Note that for each $\beta \in \barN \cap \bar{\omega}_2$ it is dense (in $\barN[\barG][\bar{H}]$) that there is a condition $\bar{r} \in \bar{\R}$ with $\beta \in {\rm dom}(\bar{r})$. It follows that if $\bar{K}$ is generic for $\bar{\R}$ over $\barN[\barG][\bar{H}]$ with $\bar{K} \ni \bar{r}_0 : = \dot{\bar{r}}^{\barG * \bar{H}}$ then $\Tilde{\sigma '} ``\bar{K}$ unions to a club in $\alpha \setminus S_G$. Since $\alpha \notin S_G$ we have that $r:=\bigcup \Tilde{\sigma '} ``\bar{K} \cup \{\alpha\}$ is a condition in $\R$ which is a lower bound on $\Tilde{\sigma '} ``\bar{K}$ and hence $r \leq \dot{r}_0^{G * H}$. Finally let $K \ni r$ be $\R$-generic over $V[G][H]$. It is now easy to check that the condition $(p, \dot{q}, \dot{r})$ and $\sigma ' \hook \barN$ collectively witness the $\infty$-subproperness of $\P * \dot{\Q} * \dot{\R}$ so we are done.
\end{proof}

We note that by the same proof adding a nonreflecting stationary set of $\mu \cap {\rm Cof}(\omega)$ for larger cardinals $\mu$ we can preserve $\infty$-$\SubPFA \hook \mu$. The following therefore holds.

\begin{theorem}
Let $2^{\aleph_0} \leq \mu \leq \lambda < \nu = \lambda^+$ be cardinals with $\mu^\omega < \nu$. Modulo the existence of a supercompact cardinal $\infty$-$\SubPFA \hook \nu + \neg \infty$-$\SubPFA \hook \mu$ is consistent.
\end{theorem}
The proof of this Theorem finishes the proof of all nonimplications involved in Main Theorem \ref{mainthm1}.

\section{Conclusion and Open Questions}
We view this paper, alongside its predecessor \cite{FS2020} as showing, amongst other things, that the continuum forms an interesting dividing line for subversion forcing: below the continuum the ``sub" plays no role as witnessed by the fact that the same non-implications can hold as those that hold for the non-sub versions. Above, it adds considerable strength to the associated forcing axioms. However, as of now we only know how to produce models of $\SCFA$ in which the continuum is either $\aleph_1$ or $\aleph_2$. The most pressing question in this area is therefore whether consistently $\SCFA$ can co-exist with a larger continuum. 
\begin{question}
Is $\SCFA$ consistent with the continuum $\aleph_3$ or greater?
\end{question}

We note here that the most obvious attempt to address this question i.e. starting with a model of $\SCFA$ and adding $\aleph_3$-many reals with e.g. ccc forcing, does not work, an observation due to the first author.

\begin{lemma}
    Suppose $\P$ is a proper forcing notion adding a real. Then $\SCFA$ fails in $V^\P$.
\end{lemma}

All that is needed about ``properness" here is that being proper implies that stationary subsets of $\kappa \cap {\rm Cof}(\omega)$ are preserved. The proof of this is standard and generalizes the proof of Lemma \ref{newseqs} above (swapping subproper for proper and removing the bound by the continuum).

\begin{proof}
Assume $\P$ is proper. Let $G$ be a $\P$-generic filter over $V$. For a contradiction, assume $\SCFA$ holds in $V[G]$.

Take a regular cardinal $\nu > 2^\omega$ in $V[G]$. In $V$, take stationary partitions $\langle A_k : k < \omega \rangle$ of $\nu \cap {\rm Cof}(\omega)$ and $\langle D_i : i < \omega \rangle$ of $\omega_1$. In $V[G]$, take a subset $r$ of $\omega$ which is not in $V$. Let $\{k(i)\}_{i < \omega}$ be the increasing enumeration of $r$.

By \cite[Lemma 7.1 of Section 4]{Jen14}\footnote{Though note, as we have cautioned throughout this text, that this lemma requires the additional assumption of $\nu > 2^\omega$, which we are assuming. See \cite[Lemma 3.29]{FuchsERR} for an error free statement.} in $V[G]$, there is an increasing continuous function $f : \omega_1 \to \nu$ such that $f[ D_i ] \subseteq A_{k(i)}$ for all $i < \omega$. Let $\alpha := {\rm sup}({\rm range}(f))$. Then, in $V[G]$ , we have that $r = \{ k \in \omega : A_k \cap \alpha$ is stationary in $\alpha \}$.

But the set $\{ k \in \omega : A_k \cap \alpha$ is stationary in $\alpha \}$ is absolute between $V$ and $V[G]$ since $\P$ is proper and hence preserves stationary subsets of ${\rm Cof}( \omega )$ points. But then $r$ is in $V$, which is a contradiction. 
\end{proof}
This shows that either $\SCFA$ implies the continuum is at most $\aleph_2$ - though given the results of this paper this seems difficult to prove by methods currently available - or else new techniques for obtaining $2^{\aleph_0} \geq \aleph_3$ are needed, which is well known to be in general an open and difficult area on the frontiers of set theory.

\bibliographystyle{plain}
\bibliography{SSFABIB}

\end{document}